\documentclass[12pt,a4paper]{amsart}
\usepackage{amsthm}
\usepackage{graphicx}
\usepackage{tikz}
\usepackage{caption}
\usepackage{subcaption}
\usepackage{lipsum}
\theoremstyle{definition}
\newtheorem{definition}{Definition}[section]

\newtheorem{remark}[definition]{Remark}

\theoremstyle{plain}
\newtheorem{theorem}[definition]{Theorem}
\newtheorem{proposition}[definition]{Proposition}
\newtheorem{lemma}[definition]{Lemma}
\newtheorem{corollary}[definition]{Corollary}

\usepackage{geometry}
\usepackage[dvipsnames]{xcolor}
\usepackage{adjustbox}

\def \G {\mathcal{G}}

\usepackage{amsthm}
\usepackage{mathtools}
\usepackage{leftindex}
\usepackage{comment}
\usepackage{bm}

\usepackage{lineno}
\usepackage{mathtools}
\mathtoolsset{showonlyrefs}

\definecolor{cornellred}{rgb}{0.7, 0.11, 0.11}

\definecolor{Royalbluee}{rgb}{0.25,0.41,0.88}
\newcommand{\defeq}{\vcentcolon=}

\DeclareMathOperator{\sech}{sech}

\usepackage[colorlinks=true,urlcolor=Royalbluee,citecolor=Royalbluee,linkcolor=Royalbluee,linktocpage,pdfpagelabels,
bookmarksnumbered,bookmarksopen]{hyperref}

\numberwithin{equation}{section}

\begin{document}

\title[NLSE with Magnetic Potential on Metric Graphs]{Nonlinear Schr\"odinger Equation \\ with Magnetic Potential on Metric Graphs}

\author[N.\,Cangiotti]{Nicol\`o Cangiotti}
\author[I.\,Gallo]{Ivan Gallo}
\author[D.\,Spitzkopf]{David Spitzkopf}

\address[N.\,Cangiotti]{Department of Mathematics\newline\indent Politecnico di Milano \newline\indent
via Bonardi 9, Campus Leonardo, 20133 Milan, Italy}
\email{nicolo.cangiotti@polimi.it}

\address[I.\,Gallo]{Dipartimento di Scienze Matematiche ``G.L. Lagrange'',\newline\indent Politecnico di Torino \newline\indent
Corso Duca degli Abruzzi, 24, 10129 Torino, Italy}
\email{ivan.gallo@polito.it}
\address[D.\,Spitzkopf]{Department of Theoretical Physics, NPI, Academy of Sciences, \newline\indent 25068 \v{R}e\v{z} near Prague, Czechia and Faculty of Mathematics and
Physics, \newline\indent Charles University, V Hole\v{s}ovi\v{c}k\'ach 2, 18000 Prague, Czechia}
\email{spitzkopf@ujf.cas.cz}



\begin{abstract}
In this manuscript, we shall investigate the Nonlinear Magnetic Schrödinger Equation on noncompact metric graphs, focusing on the existence of ground states. We prove that the magnetic Hamiltonian is variationally equivalent to a non-magnetic operator with additional repulsive potentials supported on the graph's cycles. This effective potential is strictly determined by the Aharonov-Bohm flux through the topological loops. Leveraging this reduction, we extend classical existence criteria to the magnetic setting. As a key application, we characterize the ground state structure on the tadpole graph, revealing a mass-dependent phase transition. The ground states exist for sufficiently small repulsion in an intermediate regime of masses while sufficiently strong flux prevents the formation of ground states. 
\end{abstract}

\maketitle

\section{Introduction}

The Nonlinear Schrödinger Equation (NLSE) on metric graphs has attracted significant attention in recent years as a fundamental model in mathematical physics and applied analysis \cite{Kairzhan2022, Noja2014}. This framework provides a powerful setting to investigate nonlinear dynamics in complicated geometries, motivated by the study of Bose--Einstein condensates confined in ramified traps \cite{Amico2021} and wave propagation in optical networks. These models capture phenomena such as branching, interference, and localization that have no direct analogue in Euclidean settings.

While the focusing NLSE on noncompact graphs has been extensively studied, particularly regarding the existence and stability of ground states \cite{Adami2012, Adami2015, Adami2016, Adami2017, Adami2019}, most existing results involve non-magnetic operators. In this manuscript, we extend this theory to the \emph{Magnetic} Nonlinear Schrödinger Equation:
\begin{equation}
    -\left(\frac{\textrm{d}}{\textrm{d}x} - iA\right)^2 u - |u|^{p-2}u = \lambda u\,,
\end{equation}
where $A(x)$ is a magnetic vector potential representing an external field, $2<p<6$ and Kirchhoff magnetic boundary conditions are imposed at vertices.
The introduction of a magnetic field generates rich physical phenomena and is motivated by the need to model realistic quantum transport in networks subject to external fields. As well known from the study of linear quantum graphs \cite{ExnerKovarik2015,ExnerSeba1989, KostrykinSchrader2003}, magnetic potentials on networks manifest through the \emph{Aharonov-Bohm effect} \cite{aharonov1959significance}, a purely quantum phenomenon in which electromagnetic potentials influence the quantum phase of a charged particle even in regions where the magnetic field vanishes. On a metric graph, the magnetic field can always be removed by a gauge transformation, so it is topologically trivial, i.e. locally, it has no physical effect. However, in the presence of a cycle, the magnetic potential produces a non-removable phase shift in the wavefunction. This phase shift acts as an effective repulsive mechanism that competes with the focusing nonlinearity.

The primary goal of this work is to characterize the existence of ground states (i.e. global minimizers of the energy) for the magnetic NLSE.  This study builds on the established theoretical framework for Magnetic Nonlinear Schrödinger equations \cite{CazenaveEsteban1988,EstebanLions1989,StubbeVazquez1989a,StubbeVazquez1989b}, which encompasses both the foundational analysis of magnetic effects and the concentration‑compactness approach to existence and stability of stationary states. 
Our contributions are twofold.
We prove that the magnetic NLSE on a metric graph is variationally equivalent to a standard (non-magnetic) NLSE with an additional repulsive potential supported solely on the cycles of the graph. We derive an explicit formula for this effective potential, $\Phi_\gamma(A)$, which depends on the distance between the magnetic flux and the nearest integer winding number. This reduction allows us to extend classical existence criteria \cite{Adami2015} to the magnetic setting using concentration-compactness arguments.

As a concrete application, we analyze the \textit{tadpole} graph (a ring attached to a half-line). The graph naturally emerged as a simplified model for systems such as metallic rings coupled to external electron baths via a single conducting \cite{Exner97}. With Kirchhoff conditions, it satisfies the existence criterion of a ground state for the standard nonlinear setting \cite{Adami2016}. Previous results on tadpole include the full classification of cubic solitons ($p=4$) on the tadpole graph \cite{NCF15}, the analysis of edge bifurcations for general nonlinearities \cite{NPS2015}, and the study of ground states in the presence of repulsive delta‑type vertex conditions \cite{DDL2025}.


In this work we identify an existence regime for ground states, providing a geometric condition in the case $p=4$ on the underlying graph that guarantees their occurrence. In addition, we establish a nonexistence theorem showing that ground states cannot form when the magnetic flux is noninteger and sufficiently strong.

\subsection*{Organization}
The paper is organized as follows. In Section \ref{Sec2}, we introduce the variational framework and the magnetic Sobolev spaces. Section \ref{Sec3} recalls necessary preliminaries, including Gagliardo--Nirenberg inequalities. In Section \ref{Sec4}, we prove the equivalence between the magnetic problem and the effective potential formulation. Section \ref{Sec5} establishes the general existence theory for ground states. Finally, in Section \ref{Sec6}, we perform the detailed analysis of the tadpole graph, providing explicit existence and non-existence results.

\section{The Magnetic NLSE framework}
\label{Sec2}

In this section, we formulate the NLSE on a metric graph $\mathcal{G}$ in the presence of a magnetic field. 
On a one-dimensional structure, the magnetic potential is described by a collection of real-valued functions $A = (A_e)_{e \in \mathbb{E}_\mathcal{G}}$ defined on each edge.

The stationary Magnetic Nonlinear Schrödinger Equation on the graph is given by:
\begin{equation}
    - \left( \frac{\textrm{d}}{\textrm{d}x} - i A(x) \right)^2 u  - |u|^{p-2}u = \lambda u\,,
\end{equation}
where the derivative $\frac{\textrm{d}}{\textrm{d}x}$ is replaced by the covariant derivative
\begin{equation}
    D=\frac{\textrm{d}}{\textrm{d}x}-iA\,.
\end{equation}
The associated energy functional is
\begin{equation} \label{functional}
    E_A(u, \G) = \frac{1}{2}\int_{\mathcal{G}} |Du|^2 \, \textrm{d}x - \frac{1}{p}\int_{\G} |u|^p \, \textrm{d}x\,,
\end{equation}
where $\G$ is assumed to have at least one unbounded edge. 
We seek ground states of this energy under the mass constraint
\begin{equation}
    \|u\|^2_{L^2(\G)} = \mu\,.
\end{equation}
The natural domain for this problem is the magnetic Sobolev space 
$$
H^1_A(\G) \defeq \{ u \in L^2(\G) : Du \in L^2(\G) \},
$$
and the minimization is performed over the space:
\begin{equation}
    H^1_{A,\mu}(\G) = \left\lbrace u \in H^1_A(\G) : \|u\|^2_{L^2(\G)} = \mu \right\rbrace.
\end{equation}
Consequently, the standard Kirchhoff vertex conditions are modified to preserve self-adjointness. A function $u$ in the domain of the operator must satisfy:
\begin{equation}
    \label{bc}
    \begin{cases}
        u(x) \ \text{is continuous at every vertex } v \in \mathbb{V}_\mathcal{G}\,; \\
        \sum_{e \succ v} \frac{\textrm{d} u_e}{\textrm{d}x_e}(v) - i A_e(v) u(v) = 0 \quad \forall v \in \mathbb{V}_\mathcal{G}\,,
    \end{cases}   
\end{equation}
where the sum runs over all edges $e$ incident to the vertex $v$, and derivatives are taken in the outgoing direction (see, for instance, \cite{Berkolaiko2013}).
\section{Preliminaries and Functional Framework}
\label{Sec3}
In this section, we establish the functional framework for the analysis. We first recall the non-magnetic variational framework, which will serve as the reference problem after reducing the magnetic functional to an effective potential formulation. The standard definition of a metric graph $\mathcal{G}$ as
\[
\mathcal{G}=(\mathbb{V}_\mathcal{G},\mathbb{E}_\mathcal{G}) ,
\]
where the structure consists of:
\begin{itemize}
    \item[(i)] A set of vertices $\mathbb{V}_\mathcal{G}$;
    \item[(ii)] A set of edges $\mathbb{E}_\mathcal{G}$, where elements are defined as pairs $e \defeq (v_1,v_2)$ with $v_1,v_2\in\mathbb{V}_\mathcal{G}$;
    \item[(iii)] A metric on every edge defined by the coordinate map:
    \[
    x_e: I_e \longrightarrow e, \quad \text{with } I_e \defeq [0,l_e].
    \]
\end{itemize}
A function $u:\mathcal{G}\to\mathbb{C}$ is defined as the collection of its restrictions on the edges:
\[
u \defeq (u_e)_{e\in\mathbb{E}_\mathcal{G}},
\]
where each $u_e$ is a one-dimensional function on $I_e$. The natural functional setting is given by the following spaces:

\begin{align*}
    L^p(\mathcal{G}) \defeq & \left\lbrace u:\mathcal{G}\to\mathbb{C} : \|u\|_p^p \defeq \sum_{e\in\mathbb{E}_\mathcal{G}}\int_{I_e}|u_e|^p\, \mathrm{d} x<\infty \right\rbrace;\\
      H^1(\mathcal{G}) \defeq & \left\lbrace u:\mathcal{G}\to\mathbb{C} : u \text{ is cont. on } \mathcal{G}, \text{ and }  \|u\|^2_{H^1} \defeq \sum_{e\in\mathbb{E}_\mathcal{G}}\int_{I_e} \left( |u_e|^2+|u'_e|^2 \right) \,\mathrm{d} x<\infty \right\rbrace\,.
\end{align*}

To define the operator domain, one must fix vertex conditions. The most common choice is the \emph{Kirchhoff condition}:
\begin{equation}
    \begin{cases}
        u \text{ is continuous on } \mathcal{G}\,;\\
        \sum_{e\succ v}\frac{d u_e}{d x_e}(v)=0 \quad \forall v\in\mathbb{V}_\mathcal{G}.
    \end{cases}
    \label{kirchhoff}
\end{equation}
Here, the notation $e\succ v$ indicates that the edge $e$ is incident to the vertex $v$, and the derivative is taken in the direction pointing away from the vertex. 

We are interested in the critical points of the \emph{energy functional} $E: H^1(\mathcal{G}) \to \mathbb{R}$ defined by:
\begin{equation}
E(u,\mathcal{G}) \defeq \frac{1}{2}\|u'\|_2^2 - \frac{1}{p}\|u\|_p^p\,,
\end{equation}
subject to the constraint that $u$ belongs to the space:
\[
H^1_{\mu}(\mathcal{G}) \defeq \left\{u\in H^1(\mathcal{G}): \|u\|_2^2=\mu\right\}.
\]
In the literature, the value $\mu$ is referred to as the \emph{mass}. The main goal of this framework is to solve the minimization problem:
\begin{equation}
    \mathcal{E}_{\mathcal{G}}(\mu) \defeq \inf_{u\in H^1_{\mu}(\mathcal{G})} E(u,\mathcal{G}).
\end{equation}
We say that a function $u\in H^1_{\mu}(\mathcal{G})$ is a \emph{ground state} of mass $\mu$ if $E(u,\mathcal{G})=\mathcal{E}_\mathcal{G}(\mu)$.
The starting point of the analysis of the NLSE on metric graphs is the existence of ground states at every mass for the functional $E(\cdot , \mathbb{R})$, given by the \emph{soliton}
\begin{equation}
    \phi_{\mu}(x) = \mu^{\alpha} C_p \, {\sech}^{\frac{\alpha}{\beta}}(c_p \mu^{\beta} x)
    \label{soliton} ,
\end{equation}
where $c_p$ and $C_p$ are positive constants and the exponents $\alpha$ and $\beta$ are given by
\begin{equation}
    \label{alfabeta}
    \alpha := \frac{2}{6-p}, \qquad
    \beta := \frac{p-2}{6-p}.
\end{equation}
Up to translations and multiplication by a constant phase,
the function $\phi_\mu$ is the unique
ground state at mass $\mu$ of $E$ on the line
\cite{Berestycki1983,Zakharov1972}. 

A preliminary requirement is to prove that $\mathcal{E}_{\mathcal{G}}$ is bounded from below. This depends crucially on the nonlinear power $p$. The argument relies on the \emph{Gagliardo-Nirenberg inequalities}, stated in the following proposition.
\begin{proposition}[Gagliardo-Nirenberg]
    Let $\G$ be a noncompact metric graph and let $p>2$. Then there exists a constant $K_{\G,p}>0$ such that for all $u \in H^1(\G)$:
        \begin{align}
        \tag{$\text{GN}_{\text{p}}$}
        \label{GNp}
            \|u\|^p_p &\le K_{\G,p} \, \|u\|_2^{\frac{p}{2}+1} \, \|u'\|_2^{\frac{p}{2}-1}\,,\\
        \tag{$\text{GN}_\infty$}
        \label{GNinfty}
            \|u\|_\infty &\le K_{\G,\infty} \, \|u\|_2^{\frac{1}{2}} \, \|u'\|_2^{\frac{1}{2}}\,.
        \end{align}
\end{proposition}
The proof of this proposition is standard in literature (see, e.g., \cite{Adami2012}). 

A direct consequence (see Corollary 3.4 in \cite{Adami2016}) is the following sufficient condition for the existence of ground states on graphs with infinite edges.

\begin{corollary}
Fix $2 < p < 6$ and let $\mathcal{G}$ be a noncompact graph. If there exists a function $v \in H^1_{\mu}(\mathcal{G})$ such that
\begin{equation}
  E(v, \mathcal{G})  < E(\phi_{\mu}, \mathbb{R}) \equiv \inf_{u \in H^1_\mu(\mathbb{R})} E(u, \G)\,,
  \label{criterion1}
\end{equation}
then the functional $E(\cdot, \G)$ admits a ground state at mass $\mu$.
\end{corollary}

This criterion implies that it is sufficient to find a competitor function whose energy lies below the energy of the soliton on the line. Similar existence criteria have recently been extended to NLS problems with external potentials \cite{Adami2025}; we will adopt this strategy in Section \ref{Sec5} to handle the effective magnetic potential.

\bigskip

We conclude this section with the topological definitions required to describe the magnetic flux. We recall that a \emph{cycle} is a finite sequence of connected edges where the initial and final vertices coincide, no edge is repeated, and which cannot be contracted to a point.

\begin{definition}
    For a graph $\G = (\mathbb{V}_\G, \mathbb{E}_\G)$, the \emph{cyclomatic number} (or first Betti number) $\beta_{\G}$ is the minimum number of edges that must be removed from $\mathbb{E}_\G$ to transform $\G$ into a connected tree. For a connected graph, this is given by:
    \[
    \beta_{\G} = |\mathbb{E}_\G| - |\mathbb{V}_\G| +1 \,.
    \]
\end{definition}

\begin{definition}
    An \emph{independent cycle} is identified with a subset of edges $\gamma \subset \mathbb{E}_\G$ such that removing an edge from $\gamma$ decreases the cyclomatic number $\beta_\G$ exactly by one.
\end{definition}

\begin{remark}
In the following, the "number of cycles" contributing to the energy in Eq. \eqref{Ia} refers specifically to the $\beta_\G$ independent cycles. This is the standard viewpoint for describing the effect of magnetic fluxes on metric graphs \cite{Berkolaiko2013}.
\end{remark}

The length $|\gamma|$ of a cycle is defined as the sum of the metric lengths of its constituent edges:
\[
    |\gamma| = \sum_{e \in \gamma} |e|\,.
\]

\subsection{Magnetic Sobolev space}
Magnetic Sobolev space is endowed with the norm
\begin{equation}
    \|u\|_{H^1_A(\mathcal{G})}^2 \defeq \int_{\mathcal{G}} \big( |Du|^2 + |u|^2 \big)\,\mathrm{d}x .
\end{equation}
In general, a function $u \in H^1_A(\mathcal{G})$ need not belong to $H^1(\mathcal{G})$, since continuity at the vertices is not implied.
However, the modulus $|u|$ always belongs to $H^1(\mathcal{G})$.
This follows from the diamagnetic inequality, which holds edge-wise on metric graphs.
\begin{theorem}[Diamagnetic inequality on metric graphs]
\label{diam}
Let $A \in L^2_{\mathrm{loc}}(\mathcal{G})$ and let $u \in H^1_A(\mathcal{G})$.
Then $|u| \in H^1(\mathcal{G})$ and the pointwise inequality
\begin{equation}
    \big|(|u|)'(x)\big| \le |Du(x)|
\end{equation}
holds for almost every $x$ on each edge of $\mathcal{G}$. Moreover, equality holds almost everywhere on an edge if and only if
\begin{equation}
    Du = \frac{u}{|u|}\,(|u|)' \qquad \text{a.e. on that edge},
\end{equation}
where $\frac{u}{|u|}$ is defined wherever $u\neq 0$.
\end{theorem}
\begin{proof}
The result follows by applying the classical diamagnetic inequality on each edge of $\mathcal{G}$,
viewed as an interval of $\mathbb{R}$, and summing over all edges.
\end{proof}

\section{Effective Repulsion and Variational Reduction}
\label{Sec4}

In this section, we establish that the magnetic NLSE is variationally equivalent to a non-magnetic problem augmented with a repulsive potential supported on the cycles of the graph. This reduction captures the Aharonov-Bohm effect purely through topological flux parameters.

\begin{proposition} \label{equivalence}
Let $A \in C^1(\mathcal{G}) \cap L^2_{loc}(\mathcal{G})$. Then the minimization problem
    \begin{equation}
         \mathcal{E}_{A,\mathcal{G}}(\mu) \defeq \inf_{u \in H^1_{A,\mu}(\mathcal{G})} E_A(u, \mathcal{G})
    \end{equation}
    is equivalent to the problem
    \begin{equation}
        \mathcal{I}_\mathcal{G}(\mu) \defeq \inf_{v \in H^1_{\mu}(\mathcal{G})} I_A(v, \mathcal{G})\,,
        \label{Ig2}
    \end{equation}
    where the effective functional $I_A$ is defined as
    \begin{equation} \label{Ia}
        I_A(v,\mathcal{G})  = \frac{1}{2}\int_{\mathcal{G}} |v'|^2 \mathrm{d}x + \int_{C(\mathcal{G})} \Phi_{\gamma}(A) |v|^2 \mathrm{d}x - \frac{1}{p} \int_{\mathcal{G}} |v|^p \mathrm{d}x\,,
    \end{equation}
    with $C(\mathcal{G})$ denoting the set of independent cycles and $\Phi_{\gamma}(A)$ being a scalar repulsive potential determined by the magnetic flux.
\end{proposition}
\begin{proof}
Consider a function $u \in H^1_A(\mathcal{G})$. It can be decomposed as $u(x) = v(x)e^{i \theta(x)}$ with $v = |u|$.
A direct computation yields
\begin{align*}
    \int_{\mathcal{G}} |D u|^2 &= \int_\mathcal{G} \left| v'e^{i \theta} + i \theta' v e^{i\theta} - i A v e^{i\theta}\right|^2 \mathrm{d}x \\
    &= \int_\mathcal{G} \left| v' - i(A-\theta')v \right|^2 \mathrm{d}x \\
    &= \int_{\mathcal{G}} |v'|^2 \mathrm{d}x + \int_{\mathcal{G}} (A - \theta')^2 |v|^2 \mathrm{d}x\,.
\end{align*}
Since our goal is to characterize the solutions of $\mathcal{E}_{\mathcal{G}}(\mu)$, we reformulate the problem as
\begin{equation}
    \inf_{v \in H^1_{\mu}(\mathcal{G}), \hspace{1mm} \theta \in C^1(\mathcal{G}) } E_A(u, \mathcal{G}) \,,
\end{equation}
thereby decoupling the variables and focusing on the minimization with respect to $\theta$. In particular, we seek the function that minimizes the quantity
$$
\int_\mathcal{G} (A - \theta')^2 |v|^2 \mathrm{d}x\,,
$$
while preserving the boundary conditions \eqref{bc}. 
\bigskip

Let $\gamma$ be a cycle parameterized by $[0,L]$. While the continuity conditions at the vertices require $u(0) = u(L)$, the phase function $\theta(x)$ must only satisfy the single-valued condition:
\begin{equation}
\theta(L) - \theta(0) = 2\pi m, \hspace{0.5cm} m \in \mathbb{Z}\,.
\end{equation}

Using the \emph{Lagrange Multiplier Theorem} for fixed $v$, one can find the optimal $\theta'$ subject to the periodicity constraint 
\begin{equation}
    \int_0^L \theta'(x) \mathrm{d}x = 2 \pi m.
\end{equation}
For some $\lambda \in \mathbb{R}$, we have the variation:
\begin{equation}
\nabla (A-\theta')^2 - \lambda \nabla\left(\int_0^L \theta'(x) \mathrm{d}x - 2\pi m\right ) = 0\,.
\end{equation}
Therefore we set $\theta'(x) = A(x) + \frac{1}{2}\lambda L$ for some constant $\lambda$ to be determined. The single-valued constraint becomes:
\begin{equation}
\int_0^L \theta'(x)\, \mathrm{d}x = \int_0^L \left[A(x) + \frac{1}{2}\lambda L\right ]\, \mathrm{d}x = 2\pi m\,.
\end{equation}
This gives us:
\begin{equation}
\int_0^L A(x) \mathrm{d}x + \frac{1}{2}\lambda L^2 = 2\pi m\,,
\end{equation}
which yields:
\begin{equation}
\lambda = \frac{4\pi m}{L^2} - \frac{2}{L^2}\int_0^L A(x) \mathrm{d}x \,.
\end{equation}
Consequently, substituting $\lambda$ back into the expression for $(A-\theta')$, we obtain:
\begin{equation}
     A(x) - \theta'(x) = A(x) - \left(A(x) + \frac{1}{2}\lambda L\right) = \frac{1}{L}\int_0^L A(x) \mathrm{d}x - \frac{2\pi m}{L}\,,
\end{equation} 
and the magnetic energy contribution becomes:
\begin{equation}
\int_0^L |A - \theta'|^2|v|^2 \mathrm{d}x =  \int_0^L  \left(\frac{1}{L}\int_0^L A(x) \mathrm{d}x - \frac{2\pi m}{L}\right)^2|v|^2 \mathrm{d}x\,.
\end{equation}
Finally, we define $m$ as the integer satisfying the minimization condition:
$$
(\alpha_\gamma - m)^2 = \mathrm{dist}\left(\alpha_\gamma, \mathbb{Z}\right)^2 \defeq \min_{k \in \mathbb{Z}} (\alpha_\gamma - k)^2,
$$
where we have set $\alpha_\gamma \defeq \frac{1}{2\pi} \int_{0}^L A(x) \mathrm{d} x$.
And thus we have:
\begin{equation}
    \int_0^L \left|A - \theta'\right|^2|v|^2 \mathrm{d}x = \int_0^L \frac{4 \pi^2}{L^2} \mathrm{dist}\left( \frac{\int_0^L A \mathrm{d}x}{2 \pi}, \mathbb{Z} \right)^2 |v|^2 \mathrm{d}x\, .
\end{equation}
Then, $\left(A(x) - \theta'(x)\right)$ is constant on the loop, and depends only on the size of the loop and the average magnetic flux through the cycle. 

\medskip

We can now easily generalize this analysis for a general graph with more cycles. 
Let $C(\mathcal{G})$ be the set of independent cycles of $\mathcal{G}$. Then for every $\gamma \in C(\mathcal{G})$:
\begin{equation}
    \Phi_{\gamma}(A) = \frac{4 \pi^2}{|\gamma|^2} \mathrm{dist}\left( \frac{\int_\gamma A(x) \mathrm{d}x}{2 \pi}, \mathbb{Z} \right)^2
\end{equation}
and the energy functional becomes:
\begin{equation} \label{Ia2}
     I_A(v,\mathcal{G})  = \frac{1}{2}\int_{\mathcal{G}} |v'|^2 \mathrm{d}x + \int_{C(\mathcal{G})}\Phi_{\gamma}(A) |v|^2 \mathrm{d}x - \frac{1}{p} \int_{\mathcal{G}} |v|^p \mathrm{d}x\,.
\end{equation}
Since in the definition of Eq. \eqref{Ia} only $v = |u|$ appears, for Proposition \ref{equivalence} we have $v \in H^1_{\mu}(\mathcal{G})$ and the result follows.
\end{proof}

\begin{remark}
We highlight that if \[
\int_{\gamma} A(x)\mathrm{d}x = n\cdot 2\pi\,, \qquad n\in\mathbb{Z},
\]
then there is no magnetic influence and the minimizing problem is the standard one.
\end{remark}

The functional $I_A$ can be viewed as an extension of the standard NLS energy, augmented with an effective magnetic potential term:
\begin{equation}
W(x) \defeq \sum_{\gamma \in C(\mathcal{G})} \Phi_{\gamma}(A) \chi_{\gamma}(x)\,,
\label{magneticpotential}
\end{equation}
where $\chi_\gamma$ is the indicator function of the cycle $\gamma$. Consequently, any critical point of the functional must satisfy the corresponding Euler-Lagrange equation. By standard variational arguments (see \cite{Adami2015} for further details), the minimizer $u$ satisfies the differential equation:
\begin{equation}
    -u'' - |u|^{p-2}u + W(x)u = \omega u\,,
    \label{eqdiff}
\end{equation}
on every edge, together with standard Kirchhoff conditions at the vertices. Specifically, on the non-magnetic edges $e \notin C(\mathcal{G})$, the equation is the standard stationary NLS:
\begin{equation}
    u'' + |u|^{p-2}u = -\omega u\,,
\end{equation}
while on edges $e \in \gamma$ belonging to a cycle, the equation involves the effective potential:
\begin{equation}
    u'' + |u|^{p-2}u = (\Phi_\gamma(A) - \omega) u\,.
\end{equation}

\section{Existence of ground states}
\label{Sec5}

Having reformulated the magnetic problem in terms of effective repulsive potentials, we address the question of ground state existence using concentration-compactness arguments. We seek solutions to the minimization problem:
\begin{equation}
    \label{infprob}
         \mathcal{I}_\mathcal{G}(\mu) \defeq \inf_{u \in H^1_{A,\mu}(\mathcal{G})} I_A(u, \mathcal{G})\,.
\end{equation}
We begin by stating the fundamental existence criterion.

\begin{theorem}
\label{existencecrit}
Fix $2 < p < 6$ and let $\mathcal{G}$ be a metric graph with at least one infinite edge. If there exists a function $v \in H^1_{\mu}(\mathcal{G})$ such that 
\begin{equation}
  I_A(v, \mathcal{G})  < E(\phi_{\mu}, \mathbb{R})\,,
  \label{criterion}
\end{equation}
then the functional $I_A(\cdot, \mathcal{G})$ admits a ground state at mass $\mu$.
\end{theorem}

The proof of Theorem \ref{existencecrit} relies on establishing properties of the energy functional and minimizing sequences. The strategy mirrors the arguments developed in \cite{Adami2025} for NLS with potentials.

First, we establish basic estimates.

\begin{remark} \label{phiestimate}
Since the effective potential $\Phi_\gamma(A)$ is non-negative and bounded (as explained in Section \ref{Sec4}), the following estimate holds for any $v \in H^1_\mu(\mathcal{G})$:
\begin{equation}
    0 \leq \sum_{\gamma \in C(\mathcal{G})}  \int_{\gamma}\Phi_{\gamma}(A) |v|^2 \mathrm{d}x \leq \|\Phi\|_{\infty} \cdot \mu \,.
\end{equation}
Consequently, it is always true that $I_A (v, \mathcal{G}) \geq E(v, \mathcal{G})$.
Moreover, by the definition of $\phi_\gamma$, for every $\gamma$
\begin{equation}
    \|\phi_{\gamma}\|_{L^\infty(\gamma)} \leq \min \left\{  \frac{4 \pi^2}{L^2}, \|A\|^2_{L^\infty(\gamma)} \right\}.
\end{equation}
\end{remark}
\begin{proposition} \label{estimates}
Let $\mathcal{G}$ be a noncompact metric graph.
For all $u \in H^1_{\mu}(\mathcal{G})$ satisfying the condition
\begin{equation}
    I_A(u, \mathcal{G}) \ \leq \ \frac{1}{2} \, \mathcal{I}_{\mathcal{G}} (\mu) \ < \ 0,
    \label{half inf}
\end{equation}
the following estimates hold for some constants $C_1, C_2 > 0$:
\begin{align} \label{largemu}
     C_1 \mu^{2 \beta +1} \leq &  \|u'\|_2^2 \leq C_2 \mu^{2 \beta + 1 }\,; \\
    C_1 \mu^{2 \beta +1} \leq & \|u\|^p_p  \leq C_2 \mu^{2 \beta +1}\,; \\
    C_1 \mu^{\beta +1} \leq  & \|u\|^2_{L^{\infty}(\mathcal{G})} \leq  C_2 \mu^{\beta +1}.
 \end{align}
\end{proposition}

\begin{proof}
Let $u \in H^1_{\mu}(\mathcal{G})$ satisfy hypothesis \eqref{half inf}. Thanks to Remark \ref{phiestimate}, it holds that
\begin{equation}
    \frac{1}{2} \|u'\|^2_2 - \frac{1}{p} \|u\|_p^p \leq \mathcal{I}_{\mathcal{G}} (\mu) - \int_{C(\mathcal{G})} \Phi_\gamma |u|^2 \mathrm{d}x \leq \frac{1}{2} \mathcal{I}_{\mathcal{G}}(\mu).
\end{equation}
Using the Gagliardo-Nirenberg inequalities, the proof follows with the same argument as in Lemma 2.6 of \cite{Adami2016}.
\end{proof} 

Next, we establish the concavity of the ground state energy function.

\begin{proposition}
The function $\mathcal{I}_\mathcal{G}(\mu)$ defined in \eqref{infprob} is strictly concave and strictly subadditive for $\mu > 0$.
\end{proposition}

\begin{proof}
Let $U$ be the set of normalized profiles defined by:
\begin{equation}
    U \defeq \left\{ u \in H^1(\mathcal{G}) : \|u\|_2^2= 1, \ \mu^{\frac{p}{2}} \|u\|_p^p \geq C_1 \mu^{2 \beta +1} \right\},
\end{equation}
where $C_1$ is the constant from Proposition \ref{estimates}. 
Consider the family of functions $f_u(\mu)$ obtained by scaling $u \in U$:
\begin{equation}
    f_u(\mu) \defeq I_A(\sqrt{\mu} u, \mathcal{G}) = \frac{\mu}{2} \int_{\mathcal{G}} |u'|^2 \mathrm{d}x - \frac{\mu^{\frac{p}{2}}}{p} \|u\|_p^p + \frac{\mu}{2} \int_{C(\mathcal{G})}\Phi_{\gamma} |u|^2 \mathrm{d}x.
\end{equation}
We observe that $\mathcal{I}_\mathcal{G}(\mu) = \inf_{u \in U} f_u(\mu)$. Computing the second derivative with respect to $\mu$:
\begin{equation}
  f_u''(\mu) = - \frac{p-2}{4} \cdot \mu^{\frac{p}{2} -2} \cdot \|u\|_p^p  < 0,
\end{equation}
since $p > 2$ and $\|u\|_p > 0$. Thus, $\mathcal{I}_\mathcal{G}$ is the infimum of a family of strictly concave functions, which implies it is strictly concave. Strict subadditivity follows from strict concavity and the fact that $\mathcal{I}_{\mathcal{G}}(0) = 0$.
\end{proof}

\begin{proposition} \label{minseq}
Any minimizing sequence $\{ u_n\} \subset H^1_{\mu}(\mathcal{G})$ is bounded. If $\{u_n\}$ is a minimizing sequence for $I_A(\cdot, \mathcal{G})$ and $u_n \rightharpoonup u $ weakly in $H^1(\mathcal{G})$, then one of the following alternatives holds:
\begin{enumerate}
    \item[(i)] $u_n \rightarrow 0$ in $L^{\infty}_{loc}(\mathcal{G})$ , which implies $u \equiv 0$;
    \item[(ii)] $u_n \rightarrow u$ strongly in $H^1(\mathcal{G}) \cap L^p(\mathcal{G})$ , and $u$ is a minimizer.
\end{enumerate}
\end{proposition} 

\begin{proof}
The boundedness of $\{u_n\}$ follows from Proposition \ref{estimates}. From the concentration-compactness principle on graphs, three behaviors are possible: vanishing, dichotomy, and compactness. Dichotomy is ruled out by the strict subadditivity of the energy level $\mathcal{I}_\mathcal{G}(\mu)$. Thus, only vanishing (i) or compactness (ii) can occur (see Proposition 3.4 in \cite{Adami2025} for details).
\end{proof}

\begin{proof}[Proof of Theorem \ref{existencecrit}]
Let $\{ u_n \} \subset H^1_{\mu}(\mathcal{G})$ be a minimizing sequence. By Proposition \ref{minseq}, we must rule out the vanishing case (i). 
If vanishing occurs, then $u_n \rightarrow 0$ uniformly on compact sets. Since the effective magnetic potential $\Phi_\gamma$ is supported on a compact set, the term $\int \Phi_\gamma |u_n|^2 \mathrm{d}x$ tends to 0. Consequently, the energy level of a vanishing sequence cannot be strictly below the threshold of the problem on the infinite line, $E(\phi_\mu, \mathbb{R})$.
However, the hypothesis given by the inequality \eqref{criterion} assumes the existence of a state $v$ with energy strictly below this threshold. This contradicts the vanishing scenario. Therefore, case (ii) must hold.
\end{proof}

This result extends the existence criterion to the case of bounded repulsive potentials defined on compact sets. In the following section, we apply this theorem to the tadpole graph.

\section{The tadpole graph}
\label{Sec6}

The tadpole graph consists of a graph made by one vertex and two edges. One edge is a loop that connects the vertex to itself, the other is a halfline, i.e. an edge with infinite length (see Figure \ref{tadpole}). 
It is known that with Kirchhoff conditions, it satisfies the existence criterion of a ground state proved in \cite{Adami2016}.
We recall that, in \cite{NCF15, NPS2015}, the authors classify all soliton solutions of the standard nonlinear Schrödinger equation on the tadpole graph, described via elliptic functions. In particular, if $p=4$ the ground state is given by the \textit{dnoidal} function on the ring and a piece of a soliton on the halfline.

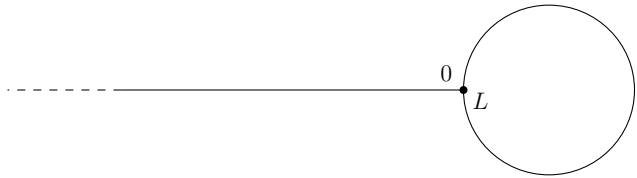
\begin{figure}[ht!]
 \centering
\begin{tikzpicture}[scale=1.5]
  \fill (0,0) circle[radius=1pt];
  \draw[dashed] (-3,0) -- (-4,0);
   \node[scale=0.75] at (0.15,-0.1) {$L$};
    \node[scale=0.75] at (-0.15,0.15) {$0$};
  \draw (-3,0) -- (0,0);
  \draw ++(0,0) arc[start angle=180, end angle=-180, radius=0.75cm];
\end{tikzpicture}
        \caption{The tadpole graph.}
        \label{tadpole}
\end{figure}
\bigskip

Using the fact that the tadpole admits a ground state for the functional $E$, we obtain the following existence results.

\begin{theorem}[Existence]
\label{existence}
    Let $\mathcal T$ be a tadpole graph and fix $\mu > 0$. If $\Phi_{\gamma}$ is sufficiently small, then there exists a ground state for $I_A(\cdot, \mathcal{T})$. 
\end{theorem}

\begin{proof}
    Fix $\mu > 0$ and let $u$ be the ground state for $E(u ,\mathcal{T})$ with energy $\mathcal{E}_{\mathcal{T}}(\mu)$. Then, the energy of $u$ can be written as 
\begin{equation*}
        I(u,\mathcal T) = \mathcal{E}_{\mathcal{T}}(\mu) + \Phi_{\gamma} \int_{\gamma} |u|^2\,.
\end{equation*}
Moreover, by the definition of ground state we have $\mathcal{E}_{\mathbb{R}}(\mu)-\mathcal{E}_{\G}(\mu) := R(\mu, \G) > 0$. 
    Therefore, if $\Phi_{\gamma}$ is sufficiently small, then
    \begin{equation}
       \Phi_{\gamma} \int_{\gamma} |u|^2 \leq R(\mu, \mathcal{T}) ,
    \end{equation}
    so the condition of Theorem \ref{existencecrit} is satisfied and a ground state exists. 
\end{proof}

\begin{remark}
In the previous proof we introduced the quantity
    \begin{equation*}
       R(\mu, \G) :=  \mathcal{E}_{\mathbb{R}}(\mu)-\mathcal{E}_{\G}(\mu).
    \end{equation*}
An explicit expression of $R(\mu, \mathcal{T})$ is available only if $p=4$. 
However, since the lowest ground state on noncompacts is the one reached in the halfline $\G = \mathbb{R}^+$ \cite{Adami2015}, one always have the general bounds
\begin{equation}
    0 \leq R(\mu, \G) \leq (2^{2\beta}-1) \theta_p \mu^{2 \beta + 1}.
    \label{Restimates}
\end{equation}
The small mass behaviour is immediate
\begin{equation*}
    R(\mu,\G) \rightarrow 0 \quad \text{if} \quad \mu \rightarrow 0.
\end{equation*}

We now show that $R(\mu, \G) \rightarrow 0$ also in the opposite regime $\mu \rightarrow \infty$.
It is known \cite{Adami2016} that every solution of the equation
\begin{equation}
    u'' - |u|^{p-2} u = \lambda u
    \label{schrodinger}
\end{equation}
satisfy the scaling law 
\begin{equation}
    u_\tau(x) = \sqrt{\tau}\, u(\tau x).
\end{equation}
Hence, as $\mu \rightarrow \infty$ the corresponding solutions become increasingly localized with characteristic length $\ell \sim t^{-1}$.
Moreover, there exists $\mathcal R > \ell$ such that $u(x) \rightarrow 0$ for $|x| > \mathcal{R}$,  In the tail region the nonlinearity is negligible and the solution behaves like solution of the linear equation $u'' = \lambda u$.
In other words, the tail decays exponentially
\begin{equation}
    u(x) \leq C e^{- c |x|}  \quad \text{for} \quad x > \mathcal{R}.
\end{equation}
By scaling, 
\begin{equation}
    u_{\tau}(x) \leq C \sqrt{\tau} e^{- c \tau|x|}  \quad \text{for} \quad x > \mathcal{R}.
\end{equation}
The energy of the tail satisfies
\begin{equation}
    \epsilon = \int_{|x|> \mathcal{R}} \frac{1}{2} |u'|^2 - \frac{1}{p} |u|^p \textrm{d}x \leq C \tau^2 e^{-c\tau \mathcal{R}} 
\end{equation}
Therefore, in the large mass regime the energy of a ground state is arbitrary closed to the energy on the line.
We conclude that
$$R(\mu,\G) \rightarrow 0 \quad \text{if} \quad \mu \rightarrow \infty.$$

Since $R(\mu, \G) \rightarrow 0 $ both as $\mu \rightarrow0$ and as $\mu \rightarrow \infty$, one should not expect ground states to appear in either extreme regime since the potential acts repulsively.
\end{remark}

We finally establish a nonexistence result, showing that for sufficiently repulsive magnetic potentials no ground state can occur. 

\begin{theorem}[Nonexistence]
Let $\mathcal{T}$ be a tadpole graph with cycle length $L$, and let $\mu>0$ be the
prescribed mass.
If $\phi_\gamma(A)$ exceeds a critical value depending on $\mu$, then
$I_A(\cdot,\mathcal{T})$ does not admit a ground state.
\end{theorem}

\begin{proof}
Assume by contradiction that $I_A$ admits a ground state $u$. Then
\begin{equation}
\mathcal I_{\mathcal{T}}(\mu)
= \min_{\|v\|_2^2=\mu}
\Big\{ E_{NLS}(v,\mathcal{T}) + \int_{\gamma} \phi_\gamma v^2\,\textrm{d}x \Big\}
\;\ge\;
\mathcal{E}_{\mathcal{T}}(\mu)
+ \min_{\|v\|_2^2=\mu} \int_{\gamma} \phi_\gamma v^2\,\textrm{d}x .
\label{minmass}
\end{equation}

A ground state must place a strictly positive amount of mass $m>0$ on the cycle.
Assume by contradiction that there exists a ground state $v$ of mass $\mu$ such that $m=0$.
Then $v=0$ on the cycle and, by continuity at the vertex, $v(0)=0$, so all the mass
lies on the half-line.
Extend $v$ to the whole line by
\[
w(x) :=
\begin{cases}
v(x), & x > 0,\\[2pt]
0, & x \le 0.
\end{cases}
\]
Then $w \in H^1(\mathbb{R})$, $\|w\|_{L^2(\mathbb{R})}^2 = \mu$, and $E(w,\mathbb{R}) = E(v,\mathcal{T}).$
Since $v$ is a ground state on the tadpole $
E(w,\mathbb{R}) \le \mathcal{E}_{\mathbb{R}}(\mu).$
On the other hand, by definition of the ground-state energy on the line,
$E(w,\mathbb{R}) \ge \mathcal{E}_{\mathbb{R}}(\mu).$
Therefore
\[
E(w,\mathbb{R}) = \mathcal{E}_{\mathbb{R}}(\mu),
\]
so $w$ is a minimizer on $\mathbb{R}$ and must coincide with the soliton.
However, the soliton on $\mathbb{R}$ is strictly positive and never vanishes, this is a contradiction, hence no ground state can have
zero mass on the cycle.

Let $m>0$ denote the minimal mass that any ground state must carry on the cycle. Assume that $\phi_{\gamma} >0 $, then $\int_{\gamma} \phi_\gamma |v|^2 \geq \phi_{\gamma} m$.
Combining this with \eqref{minmass} we obtain
\begin{equation}
\phi_\gamma(A)\, m \;<\; R(\mu,\mathcal{T}),
\label{nonexistencecondition}
\end{equation}
If $\phi_\gamma(A)$ is larger than the critical value for which
\eqref{nonexistencecondition} fails, then the inequality cannot hold.  
This contradicts the existence of a ground state, and therefore no ground state can
exist in this regime.
\end{proof}

\subsection{The hyperbolic secant competitor}
\label{Sec6.1}

Having established the general existence theory, we now turn to the concrete case of the tadpole graph with $p=4$. The key step is to construct an admissible test function that beats the energy of the soliton on the real line. Drawing inspiration from the standard case, we build our competitor from hyperbolic secant profiles.  The computation yields an explicit inequality involving the magnetic potential that assure the existence of ground states, in particular this is true for an intermediate regime mass (see Figure \ref{fig:competitor2}).

\begin{lemma}
Let $\mathcal{T}$ be the tadpole graph with loop length $2L>0$, and let $\mu>0$ be the prescribed mass. Consider the nonlinear Schr\"odinger energy functional $I_A(\cdot,\mathcal{T})$  with $p=4$.
Let $m=m(\mu)>0$ be the unique solution of the mass relation
\begin{equation}
\mu = 2m\big(1+\tanh(mL)\big).
\end{equation}
If the magnetic potential satisfies
\begin{equation}
    \label{magneticexistence}
    \phi_{\gamma}(A) \;\leq\; m(\mu)^2\,\frac{1 + 3\tanh\!\big(m(\mu)L\big)}{\sinh\!\big(2m(\mu)L\big)}\,,
\end{equation}
then the energy functional on $\mathcal{T}$ admits a ground state of mass $\mu$.
\end{lemma}

\begin{proof}
We use the competitor $u_\mu$ constructed from hyperbolic secant profiles on the loop and on the half-line (see Figure \ref{fig:competitor}):

\[
u_\mu(x)=
\begin{cases}
\sqrt{2}\,m\,\sech\!\big(m(y+L)\big), & y\in[0,\infty),\\
\sqrt{2}\,m\,\sech(mx), & x\in[-L,L].
\end{cases}
\]

\begin{figure}[ht!]
 \centering
 \includegraphics[height = 4 cm ]{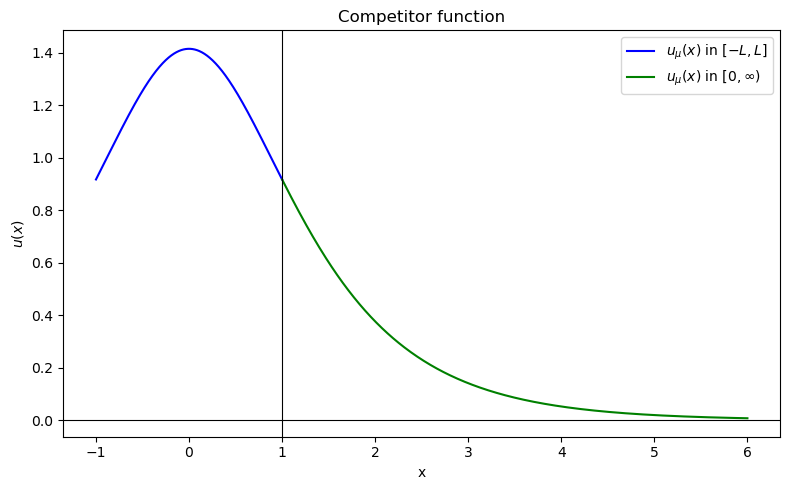}
        \caption{Competitor function $u_\mu(x)$ with $L=1$ and $\mu=1$}
        \label{fig:competitor}
\end{figure}

A direct computation using $\int \sech^2(\xi)\,d\xi=\tanh(\xi)+C$ gives the mass identity

\[
\mu = 2m^2\!\int_{-L}^L \sech^2(mx)\,\textrm{d}x
      +2m^2\!\int_0^\infty \sech^2\!\big(m(y+L)\big)\,dy
    =2m\big(1+\tanh(mL)\big),
\]

which determines $m=m(\mu)$ uniquely .

For the energy, we use the identity

\[
\int\!\left[\sech^2(\xi)\tanh^2(\xi)-\sech^4(\xi)\right]d\xi
=
\frac{2\tanh^3(\xi)-1}{3}-\tanh(\xi)+C.
\]

Evaluating on the loop and the half-line yields the nonlinear contribution

\[
E_{NLS}(u_\mu, \mathcal T)
=
\frac{1}{3}m^3\big(2T^3-3T-1\big),
\qquad T=\tanh(mL).
\]

The magnetic term contributes $\Phi_{\gamma}2mT$, so the total energy is

\[
I(u_\mu)
=
\frac{1}{3}m^3\big(2T^3-3T-1\big)+2\phi_{\gamma}mT.
\]

The soliton on $\mathbb{R}$ with the same mass $\mu$ has energy

\[
E_{\mathbb{R}}(\mu)
=
-\frac{1}{12}m^3(1+T)^3,
\]

where the same mass relation links $m$ and $\mu$.  
Requiring $I(u_\mu, \mathcal{T})\leq E_{\mathbb{R}}(\mu)$ gives

\[
4\phi_{\gamma}T \leq m^2(-3T^3-T^2+3T+1),
\]

which simplifies to

\[
\phi_{\gamma}(A) \leq 
 m^2\,\frac{1+3\tanh(mL)}{\sinh(2mL)}.
\]
Under this condition, the competitor $u_\mu$ has strictly lower energy than any soliton on $\mathbb{R}$, and for Theorem \ref{existencecrit} a ground state exists.
\end{proof}

The explicit relation \eqref{magneticexistence} obtained in the Lemma is illustrated in Figure \ref{fig:competitor2}, where the intermediate mass regime for which the inequality is satisfied is highlighted.

\begin{figure}[ht!]
 \centering
 \includegraphics[height = 4 cm ]{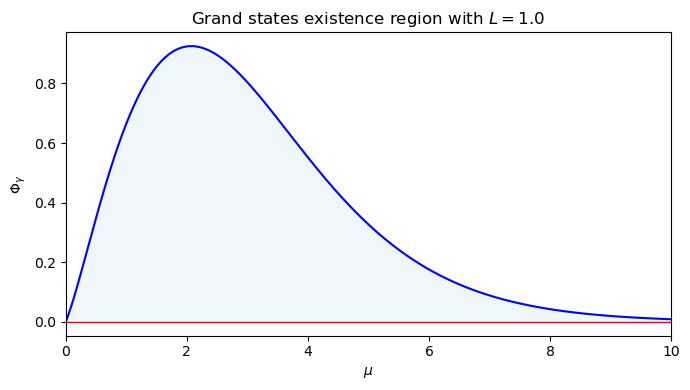}
        \caption{Parameter region in the $(\mu, \phi_\gamma)$-plane where the condition of the inequality \eqref{magneticexistence} holds with $L=1$.}
    \label{fig:competitor2}
\end{figure}

\end{document}